\documentclass[12pt]{article}
  \usepackage[centertags]{amsmath}
  \usepackage{amsfonts}
  \usepackage{amssymb}
  \usepackage{amsthm}
  \usepackage{newlfont}

  \newlength{\defbaselineskip}
  \newcommand{\setlinespacing}[1]%
                               {\setlenght{\baselineskip}{#1 \defbaselineskip}}

  \newcommand{\dis}{\displaystyle}

  \theoremstyle{plain}

  \newtheorem{thm}{Theorem}[section]
  \newtheorem{cor}[thm]{Corollary}
  
  \newtheorem{pro}[thm]{Proposition}
  \newtheorem{rem}[thm]{Remark}
  
  \theoremstyle{definition}
  \newtheorem{defi}[thm]{Definition}
  \newtheorem{exm}[thm]{Example}
  
  \newtheorem{notation}[thm]{Notation}
\numberwithin{equation}{section}

\begin{document}
\begin{center}
{\bf  Limit Theorems for Numerical Index}
\end{center}
\vspace{.15 cm}
\begin{center}
\small{Asuman G\"{u}ven AKSOY and Grzegorz LEWICKI}
\end{center}

  \date{May 20, 2010}
\mbox{~~~}\\
\mbox{~~~}\\
\small\mbox{~~~~}{\bf Abstract.} {\footnotesize  We improve upon on a limit theorem for numerical index for large classes of Banach spaces including vector valued $\ell_p$-spaces and $\ell_p$-sums of Banach spaces where\\ $1\leq p \leq \infty$. We first prove $ n_1( X) = \displaystyle \lim_m n_1( X_m)$ for a modified numerical index $n_1(\, .\, )$. Later, we establish if a norm on $X$ satisfies the local characterization condition, then $n(X) = \displaystyle\lim_m n(X_m).$  We also present an example of a Banach space where the local characterization condition is satisfied.


\section{ Introduction}
Let $X$ be a Banach space over $\mathbb{R}$ or $\mathbb{C}$. We write $B_{X}$ for
the closed unit ball and $S_{X}$ for the unit sphere of $X$.
 The dual space is denoted by $X^{*}$ and the Banach algebra of all
continuous linear operators on $X$ is denoted by $B(X)$.
\begin{defi}
The \textit{numerical range} of $T\in B(X)$ is defined by
$$W(T)= \{ x^{*}(Tx)  :~x\in S_{X},~x^{*}\in S_{X^{*}},~x^{*}(x)=1\}\cdot$$
The \textit{numerical radius} of $T$ is then given by
$$\nu (T)=\sup\{\vert \lambda\vert : ~\lambda\in W(T)\}\cdot$$
\end{defi}
Clearly, $\nu(.)$ is a semi-norm on $B(X)$ and $ \nu(T) \le \Vert T\Vert$ for
all $T\in B(X)$.
 The \textit{numerical index} of $X$ is defined by
$$n(X)=\inf\{\nu(T) :~ T\in S_{B(X)}\}\cdot$$
Equivalently, the numerical index $n(X)$ is the greatest constant $k
\geq 0$ such that  $k\|T\| \leq \nu(T)$ for every $T \in B(X)$. The concept of numerical index was first introduced by Lumer
\cite{lg} in 1968. Since then, much attention has been paid to this
 equivalence constant between the numerical radius and the usual
norm in the Banach algebra of all bounded linear operators of a
Banach space. It is known that  $0 \leq n(X) \leq 1$ if $X$ is a real space, and $\displaystyle\frac{1}{e}\leq n(X) \leq 1$ if $X$ is a complex space. Furthermore, $n(X)> 0$ if and only if $\nu (.)$ and $\|\cdot\|$ are equivalent norms. Calculation of numerical index for some classical Banach spaces can be found in  \cite{bff-dj1} and \cite{bff-dj2}. For more recent results we refer the reader to
\cite{aag-cbl}, \cite {aag-gl}, \cite{aga-ed-kham}, \cite{ee},
\cite{fc-mm-pr}, \cite{gke-rdkm}, \cite{lg-mm-pr}, \cite{mm} and \cite {mm-jm-mp-br}. In \cite{ee-2} it is shown that $$ n(\ell_p) = n(L_p[0,1])$$ for a fixed $1<p< \infty$. In the same paper it is also established that $n(\ell_{p}^{m})\neq 0$ for finite $m$ in the real case. In \cite{aga-ed-kham} numerical index of vector-valued function spaces is considered  and a proof of $$ n(L_p(\mu, X)) = \displaystyle \lim_m n( (\ell_p^m (X))$$ is provided for a Banach space $X$ and for $1 \leq p < \infty$. Furthermore, it was recently proven in \cite{mm-jm-mp} that for $p\neq 2$ and $\mu$ any positive measure one has  $n(L_p(\mu)) > 0$ in the real case.
 In this paper we first obtain the above type of limit theorem for a class of Banach spaces including vector valued $\ell_p$ or $L_p$ spaces. We do this in two steps. First, we modify the definition of numerical index and show that $$ n_1( X ) = \displaystyle \lim_m n_1( X_m)$$  where $n_1$ is the modified numerical index, and verify that $n_1= n$ for all the examples of spaces considered in this paper. However, our main result is an improvement of limit theorem presented in \cite{mm-jm-mp-br}. The study of numerical index of absolute sums of Banach spaces is given in \cite{mm-jm-mp-br}, where under suitable conditions it is shown that the numerical index of a sum is greater or equal to the limsup  of the numerical index of the summands. (See Theorem 5.1 of \cite{mm-jm-mp-br}.) In this paper, we show the liminf of the numerical index of the summands is greater or equal to the numerical index of the sum if the Banach space satisfies a condition called the local characterization condition(LCC).
 We establish : if a norm on $X$ satisfies the local characterization condition, then $$n(X) = \displaystyle\lim_m n(X_m).$$  We also provide examples of spaces where (LCC) is satisfied.

\section{Main Results }
\begin{notation}
\label{nota}
Let $X$ be a Banach space and let for $m\in \mathbb{N}$, $X_m$ denote a closed  (not necessarily finite dimensional) subspaces of $X$. Throughout this paper, unless otherwise stated, we assume:
\begin{enumerate}
\item $X = \dis \overline{\bigcup_{m=1}^{\infty}X_m}$ and  $X_{m} \subset X_{m+1}$ for all $m\in \mathbb{N}$.
\item For any $m\in \mathbb{N}$, there exists $P_m \in \mathcal{P}(X, X_m)$ with $\|P_m \| = 1$, where $\mathcal{P}(X,X_m)$ is the set of all linear projections continuous with respect to the operator norm.
\item For any $x\in X $ and  $j\geq m$ we have $P_m(P_j x)=P_m x$.
\end{enumerate}
\end{notation}
\begin{defi}
\label{modifydef}
Let $X$, $X_m$ and $P_m$ be defined as above. For $L \in \mathcal{L}(X_m)$, \emph{modified numerical index} $n_1(L)$ is defined as
$$ n_1(L) = \mbox{sup} \{ |x^* P_jLP_jx|: \,\, j\in \{1,2\ldots, m\},  \,\, x^* \in S_{X^*}, x\in S_X, x^*(x)=1 \}$$
Furthermore,
 $$
n_1(X_m)=\mbox{inf} \{n_1(L): \,\,\, L\in \mathcal{L}(X_m),\,\,\, \|L\|=1 \,\}.
$$
However, if $L \in \mathcal{L}(X)$, then the above definition takes the form
$$
n_1(L) = \mbox{sup} \{ |x^* P_jLP_jx|: \,\, j\in \mathbb{N},  \,\, x^* \in S_{X^*}, x\in S_X, x^*(x)=1\}
$$
and similarly
$$
n_1(X)=\mbox{inf} \{n_1(L): \,\,\, L\in \mathcal{L}(X),\,\,\, \|L\|=1 \,\}.
$$
\end{defi}
\begin{thm}
\label{modifythm}
Let $X$ and $X_m$ be as in the Notation 2.1. Then $$ n_1( X) = \displaystyle \lim_m  n_1( X_m).$$
\end{thm}
\begin{proof} Consider $L\in \mathcal{L}(X_m)$ where $m\in \mathbb{N}$ is fixed. Assume $\|L\|=1$ and define
$$
L_m =L\circ P_m \in \mathcal{L}(X).
$$
Since $\|P_m\|=1$, we have $\|L_m \| =\|L\|$.  We first claim that $n_1(L)=n_1(L_m)$. Let $j \in \{1,...,m\}.$ Observe that $P_jx \in X_m$ since $X_j \subset X_m.$ By definition
$$
n_1(L) = \mbox{sup} \{ |x^* P_jLP_mP_jx|: \,\, j=\{1,2\ldots, m\} \,\, x^* \in S_{X^*}, x\in S_X, x^*(x)=1\}
\leq n_1(L_m).
$$
Now, suppose $j> m$, then by  Notation(\ref{nota}), Part 3 above we know $P_m(P_j x)=P_m x$. Thus
$$ |x^* P_j L P_m P_j x |= |x^* P_jL P_m x|=|x^* L P_m x|= |x^* P_m L P_m x| ,
$$
and therefore
$$
n_1(L_m)=\mbox{sup} \{ |x^* P_j L P_j x|:\,\, j \in \{1,...,m\},  \,\, x^* \in S_{X^*}, x\in S_X,x^*(x)=1\}= n_1(L)
$$
 as claimed.
\\
Now, since we have $n_1(L) =n_1(L_m) \geq n_1(X)$ for any $L \in  \mathcal{L}(X_m) $ with $\|L \|=1$ and for any $m\in \mathbb{N}$, taking infimum over $L$ yields $n_1(X_m)\geq n_1(X)$ which in turn implies $\liminf_m n_1(X_m)\geq n_1(X)$.
\\
To prove $\limsup_m n_1(X_m)\leq n_1(X)$, we start with
$S \in \mathcal{L}(X)$ with $\|S\| =1$ and for $m\in \mathbb{N}$ define $S_m = P_m \circ S|_{X_m} \in \mathcal{L}(X_m)$. We know $\|S_m\|\leq \|S\|$, but we claim $\|S_m\|\rightarrow \|S\|=1$. To show this let $\epsilon > 0$ and $x\in X$ with $\|x\|=1$, note that $\|S x\|> \|S\|-\epsilon$, but need to show
$$
\|P_m S P_m x \|\rightarrow \|Sx\|\,\,\,\,\mbox{as}\,\,\, m\rightarrow\infty .
$$
Consider,
$$
|\|P_m S P_m x \|- \|Sx\||\leq \|P_m S P_m x - Sx\|\leq \|P_m S P_m x-P_m S x \|+\|P_m S x -S x\|
$$
and therefore
$$
|\|P_m S P_m x \|- \|Sx\||\leq \|P_m\| \|S(P_mx)-Sx \|+ \|P_m(Sx)-Sx \|.
$$
Now, for any $ z \in  \dis \bigcup_{m=1}^{\infty}X_m$ , $P_m z\rightarrow z$. Since $\|P_m\| =1$ and $X = \dis \overline{\bigcup_{m=1}^{\infty}X_m}$, thus applying the Banach -Steinhaus theorem yields, $\|P_m(Sx)-Sx \|\rightarrow 0$ and continuity of $S$ implies $\|S(P_mx)-Sx \|\rightarrow 0$. Since we proved $\displaystyle\mbox{lim}_{m} \| P_m S P_m x \|\rightarrow \|Sx\|\geq \|S \|-\epsilon$, we also have $\mbox{lim inf}_m\|P_m S \mid_{X_m}\|\geq \|S\| -\epsilon$, for an arbitrary $\epsilon$. It follows that
$$
\|P_m S \mid_{X_m}\| \leq \|S\mid_{X_m}\| \leq \|S\|
$$
and thus $\|S\| \geq\mbox{lim sup}_m \|P_m S \mid_{X_m}\|$. Combining together,
$$
\mbox{lim inf}_m\|P_m S \mid_{X_m}\|\geq \|S\|\geq \mbox{lim sup}_m \|P_m S \mid_{X_m}\|
$$
this concludes the proof of the claim that $\|S_m\|\rightarrow \|S\|=1$. Consider, $S\in \mathcal{L}(X)$ with $\|S\|=1$, and $n_1(S_m)=n_1(P_m S\mid_{X_m} )$ where
$$
n_1(P_m S\mid_{X_m}) =\mbox{sup}\{ | x^* P_j P_m S\mid_{X_m} P_j x|:\,\, j\in \{1,2\ldots, m\}  \,\, x^* \in S_{X^*}, x\in S_X \,\, x^*(x)=1\}.
$$
By assumption we have $P_j(P_m z)=P_jz $ when $m\geq j$ for all $z\in X$, therefore,
$$
|x^* P_j P_m S\mid_{X_m} P_j x|=|x^* P_j  S P_j x|\,\,\,\mbox{and}\,\,\, j\in \{1,2\ldots, m\}
$$
and we have:
$$
n_1( P_m S\mid_{X_m}) \leq \mbox{sup}\{ | x^* P_j S P_j x| :\, j\in \mathbb{N},x\in S_X, x^* \in S_{X^*}, x^*(x)=1\}=n_1(S).
$$
From the above argument $\mbox{lim sup}_m n_1(S_m) \leq n_1(S)$, and since $\|S_m\| \rightarrow \|S\|=1$, it follows that $\mbox{lim sup}_m  n_1(\displaystyle\frac{S_m}{\|S_m\|}) \leq n_1(S)$ and therefore, $\mbox{lim sup}_m n_1(\displaystyle \frac{S_m}{\|S_m\|}) \geq \mbox{lim sup}_m n_1(X_m)$. Since we have shown that
$$
\mbox{lim sup}_m n_1(X_m) \leq n_1(S) \,\,\,\mbox{for any }\,\,\,\, \|S\| =1, \,\,\, S \in \mathcal{L}(X),
$$
taking infimum over $S$ yields $$ \mbox{lim sup}_m n_1(X_m)\leq n_1(X),$$ which completes the proof.
\end{proof}
\begin{rem}
Let
$$
F = \{ \{x_n\}, x_n \in \mathbb{R} \mbox{ and }x_n = 0 \mbox{ for } n \geq m \mbox{ depending on } \{ x_n\} \}.
$$
Let $\|\cdot \|$ be any norm on $F$ satisfying
$$
\|(x_1,x_2,...,)\| \leq \|(y_1,y_2,...)\| \mbox{ provided } |x_i| \leq |y_i| \mbox{ for } i \in \mathbb{N}.
$$
Let $ X $ denote the completion of $F$ with respect to $ \| \cdot \|.$  If for $m \in \mathbb{N}$
$$
X_m = \{ \{x_n\}, x_n \in \mathbb{R} \mbox{ and }x_n = 0 \mbox{ for } n > m \}
$$
and let $P_m : X \rightarrow X_m$ be defined as
$$
P_m (x_1,...,x_m,x_{m+1},...) = (x_1,...,x_m,0,...).
$$
It is easy to see that the above defined $X_m$ and $P_m$ satisfy the assumptions of Notation(\ref{nota}).
Observe that classical sequence spaces like $l_p$-spaces, Musielak-Orlicz sequence spaces (in particular Orlicz spaces) and Lorentz
sequence spaces can be constructed in the above presented manner. Hence Theorem(\ref{modifythm}) can be applied in these cases.
However, from Example 5.4 of \cite{mm-jm-mp-br}, we know that, Theorem(\ref{modifythm}) does not hold for the classical numerical index.
\end{rem}
Next, we show  a class of spaces for which analogous result to Theorem(\ref{modifythm}) holds for the classical numerical index.
\begin{defi}
Let $X, X_m$ and $P_m$ be as in the Notation(\ref{nota}) and $\|\,.\, \|_X $ denotes the norm on $X$. We say the norm $\|\,.\, \|_X $ satisfies the \emph{Characterization Condition}(CC) if and only if  for any $x \in X$ , with $\|x \|_X =1$ and $m \in \mathbb{N}$, if $x^*$ is a norming functional for $x$, then there exists a constant $b_m(x)$ such that
$b_m(x)\, x^*\mid_{X_m}$ is a norming functional for $P_m x$.
\end{defi}
 Above definition is motivated by the space $X= \ell_p$ with $1< p < \infty$. For $x\neq 0$ and $x \in \ell_p$,
form of the norming functional is $x^* = \displaystyle \frac{(|x_i|^{p-1} sgn(x_i))}{\|x\|_{p}^{p-1}}$ and clearly
$$
x^*\mid_{X_m} = \displaystyle \frac{(|x_i|^{p-1} sgn(x_i))}{\|x\|_{p}^{p-1}} \,\,\,\mbox{where}\,\,\,i\in\{1,2,\ldots ,m\}
$$
and the norming functional for $P_mx$, $(P_mx)^*$ takes the form
$$
(P_m x)^* = \displaystyle \frac{(|x_i|^{p-1} sgn(x_i))}{\|P_m x\|_{p}^{p-1}}\,\,\,\mbox{where}\,\,\, b_m(x) =\displaystyle \frac{\|x\|_{p}^{p-1}}{\|P_m x\|_{p}^{p-1}} .
$$
The above (CC) is also satisfied for norms of $\ell_1$ and $c_0$.
 The next theorem states that if the characterization condition is  satisfied then modified numerical radius is
equal to the classical one. Thus it becames important to give examples of spaces besides $\ell_p$ where characterization condition is satisfied. We present another example after the following theorem.
\begin{thm}
Let $X, X_m$ and $P_m$ be as in the Notation(\ref{nota}) and assume that $\|\cdot\|_X $  satisfies the characterization condition given above. Then, for any $L\in\mathcal{L}(X_m)$  and for $m \in\mathbb{N}$, $$n_1(L)= \nu(L).$$ Similarly, for any $L\in\mathcal{L}(X)$ we also have $n_1(L)= \nu(L)$. Furthermore,
$$
n(X) =\displaystyle\lim_m n(X_m).
$$
\end{thm}
\begin{proof}
Take $L\in \mathcal{L}(X_m)$ with $\|L\|=1$, then
$$
\nu(L)= \mbox{sup} \{|x^*L \,x|:\,\, x\in S_{X_m},\,\, x^*\in S_{X_m^*},x^*(x)=1\}.
$$
However, $|x^* P_m L P_m x |= | x^* L\,x|$ implies that
$$ \nu(L) \leq \mbox{sup} \{ |x^* P_j L P_j x |:\,\,j\in\{1,2,\ldots,m\},\,\, x\in S_{X_m},\,\, x^* \in S_{X_m^*},\, \,x^*(x)=1\}
$$
$$
\leq\mbox{sup} \{ |x^* P_j L P_j x |:\,\,j\in\{1,2,\ldots,m\}\,\, x\in S_{X},\,\, x^* \in S_{X^*},\, \,x^*(x)=1\}= n_1(L).
$$
To prove the other inequality, assume $L \in \mathcal{L}(X_m)$ with $\|L\|=1$ and that $\nu(L) < n_1(L).$
By definition, there exists $x^*\in S_{X^*}$, $x\in S_{X},$ $x^*(x)=1$ and $j\in\{1,2,\ldots,m\}$ such that
$$ \nu(L) < |x^* P_j L P_j x |=|(x^*\mid_{X_j}\circ P_j\mid_{X_m}) L P_j x |\leq | b_j(x)x^*\mid_{X_j} P_j\mid_{X_m} L (\displaystyle\frac{P_jx}{\|P_jx\|})|\leq \nu(L),
$$
since $ b_j(x)\,x^*\mid_{X_j}\circ P_j\mid_{X_m} $ is a norming functional for $P_j x$ in $X_m$, thus we reached to a contradiction.
Note that to obtain the last inequality in the above equation we use the facts $b_j(x)\, x^* \mid_{X_j} = (P_j x)^*$ and \,\,\,$|\displaystyle \frac{b_j(x)}{\|P_j x\|}|\geq 1. $
\\
To prove $n_1(L)= \nu(L)$ when $L \in \mathcal{L}(X)$ with $\|L\|=1$, let $\epsilon >0$ be fixed, then for some $x^*\in S_{X^*},\, x\in S_{X}$ with $x^*(x)=1$, we have $|x^*L x|> n(L)-\epsilon$. However,
$$
|x^*L x|= \mbox{lim}_m | x^* P_mL P_mx |\leq \mbox{sup} \{| x^* P_mL P_mx |:\,\,m\in\mathbb{N},\, x^*(x)=1, \,\,\|x\|=\|x^*\|=1 \,\,\}=n_1(L)
$$
Thus we have
$$
\nu(L)-\epsilon \leq n_1(L)\,\,\,\mbox{for any }\,\,\, \epsilon>0 .
$$
To show the reverse inequality, assume $L \in \mathcal{L}(X)$ with $\|L\|=1$ with $\nu(L) < n_1(L)$.
Then $\nu(L) < | x^* P_mL P_mx |$ for some $x^*\in S_{X^*},\, x\in S_{X}$ with $x^*(x)=1$   and $m\in \mathbb{N}$, since $\frac{|b_m(x)|}{\|P_m x\|}\geq 1$, we have,
$$ \nu(L) < | (x^*\mid_{X_m}\circ P_m)L P_mx |\leq | (b_m(x)x^*\mid_{X_m} \circ P_m) L (\displaystyle\frac{P_mx}{\|P_mx\|})\leq \nu(L).
$$
Since $ b_m(x)x^*\mid_{X_m}\circ P_m$ is a norming functional for $P_m x$ in $X$,
we again reach to a contradiction. Therefore $\nu(L) =n_1(L)$ holds true.
\\
Since for any $L\in \mathcal{L}(X_m)$, with $\|L\|=1$ we have $\nu(L)=n_1(L)$ taking infimum over $L$ yields the equality $n(X_m)=n_1(X_m)$ and similarly for  $L\in \mathcal{L}(X)$, with $\|L\|=1$ we have $\nu(X)=n_1(X)$. By Theorem (\ref{modifythm}) and combining all of these equalities we have
$$
n(X) =\displaystyle\lim_m n(X_m).
$$
\end{proof}

\begin{exm}
\label{exCC}
Fix $p \in (1,\infty).$ Let $X =(\displaystyle \oplus_{i \in \mathbb{N}} X_i)_p$ be the direct $l^p$-sum of Banach spaces $(X_i, \parallel . \parallel_i)$, defined as
$$
X  = \{(x_1,\ldots, x_n, \ldots): \,\,\, x_i \in X_i \mbox{ and } \displaystyle\sum_{i=1}^{\infty} ( \parallel x_i \parallel_i)^p < \infty \}.
$$
Clearly, the norm $x \in X$ is
$$
\parallel x \parallel =  (\displaystyle\sum_{i=1}^{\infty} ( \parallel x_i \parallel_i)^p)^{1/p}
$$
and in case $X_i=X_j$
 for all $i,j \in \mathbb{N}$, then $X= \ell_p(X)$.
Next, we consider spaces $Z_m= X_1 \oplus \ldots \oplus X_m$ and the projections
$$
P_m(x_1, \ldots , x_m,x_{m+1}, \dots ) = ( x_1,x_2, \ldots x_m, 0, 0, \ldots).
$$
Note that all conditions in Notation(\ref{nota}) are satisfied for $X$ ,$Z_n$ and $P_m$. To show that characterization condition is satisfied for the norm on $X$, note that for any $x \in X\setminus\{0\} $ the norming functional has the form
$$
x^* = \displaystyle \frac{\left( \parallel x_i \parallel_i^{p-1} x_i^* ( . )\right )_{i= 1}^{\infty}}{\displaystyle\sum_{i=1}^{\infty} ( \parallel x_i \parallel_i)^p)^{\frac{p-1}{p}}}
$$
where $x_i^* \in X_i^*$ is a norming functional for $x_i \in X_i$. Setting $C = \displaystyle\sum_{i=1}^{\infty} ( \parallel x_i \parallel_i)^p) ^{\frac{p-1}{p}}$, to see  $\parallel x^* \parallel \leq 1$, let $y\in X$ be an element with $\parallel y \parallel =1$, then
$$
|x^* (y)| = | \displaystyle \frac{\sum_{i=1}^{\infty} \left( \parallel x_i \parallel_i\right)^{p-1} x_i^*(y_i)}{C}| \leq  \frac{1}{C}\displaystyle\sum_{i=1}^{\infty}
 \left( \parallel x_i \parallel_i)^{p-1}\right) |x_i^*(y_i)|.
$$
Applying the H\"{o}lder inequality with conjugate pairs $p$ and $q$ imply :
$$ |x^* (y)| \leq \frac{1}{C} \left[\displaystyle (\sum _{i=1}^{\infty} \parallel x_i\parallel_{i}^{p-1})^q\right]^{\frac{1}{q}} . \left[\displaystyle\sum_{i=1}^{\infty}\parallel y_i \parallel_i^p\right]^{\frac{1}{p}}.
$$
Since $q =\displaystyle \frac{p}{p-1}$ and $\parallel y \parallel =\left[\displaystyle\sum_{i=1}^{\infty}\parallel y_i \parallel^p\right]^{\frac{1}{p}} = 1$ we have $|x^*(y)| \leq 1$.
It is easy to see that $x^*$ is a norming functional for $x$ because
$$
x^* (x)= \frac{1}{C}\displaystyle \sum_{i=1}^{\infty} \parallel x_i \parallel_{i}^{p-1} x_i^*(x)= \displaystyle \frac{\parallel x \parallel^p}{\parallel x \parallel^{p-1}}= 1.
$$
Furthermore, from
$$ (P_mx)^* = \displaystyle\frac{\left( \parallel x_i \parallel_i^{p-1} x_i^*(.)\right)_{i=1}^{m}}{ \left(\displaystyle \sum_{i=1}^{m} \parallel x_i \parallel_i^p \right)^{\frac{p-1}{p}}}
$$
and  writing $x^*\mid_{Z_m}$ we obtain that $b_m(x)= \displaystyle \frac{\parallel x \parallel^{p-1}}{\parallel P_mx\parallel^{p-1}}$.
\end{exm}
 Next we define characterization condition locally and prove that for a Banach space $X$ with the local characterization condition
we also have the limit theorem for the classical numerical index.
\begin{defi}
\label{defLCC}
Let $X$ be a Banach space and $X_1\subset X_2\subset \cdots \subset X$ be its subspaces such that $X=
\displaystyle\overline{ \bigcup_{m=1}^{\infty}X_m}$.
Suppose for any $m\in \mathbb{N}$ there exists $P_m \in \mathcal{P}(X_{m+1}, X_m)$ with $\|P_m \| = 1$. We say the norm on $X$, $\parallel . \parallel_{X}$ satisfies the \emph{Local Characterization Condition} (LCC) if and only if for any $x \in X_{m+1}$ with $\parallel x \parallel =1$, if $x^* \in S_{X_{m+1}^*}$ is a norming functional for $x \in X_{m+1}$, then there exist a constant $b_m(x)\in \mathbb{R}$ such that $b_m(x) x^*\mid _{X_m}$ is a norming functional for $P_m x$ in $X_m^*.$
\end{defi}
We start by investigating some consequences of (LCC).
\begin{pro}
For a fixed $m\in \mathbb{N}$ and $L \in \mathcal{L}(X_m)$,define a sequence
$$
w_m(L)= \nu(L),\,\, w_{m+1}(L) = \nu(L\circ P_m), \ldots, w_{m+j} (L) = \nu (L \circ Q_{m,j}),
$$
where $Q_{m,j}= P_m \circ \ldots \circ P_{m+j-1}.$
If the norm on $X$ satisfies (LCC) then
$ \nu(L)=w_m(L)=w_{m+j} (L)$ for $j=1,2,\ldots$.
\end{pro}
\begin{proof}
Since $ X_m \subset X_{m+1}$ for any $m \in \mathbb{N},$ it is easy to see that $w_{m+j} (L)$ is an increasing sequence with respect to $j$, since
$$
w_{m+j}(L)= \sup\{ |x^* L \circ Q_{m,j} x |:\,\,x\in S_{X_{m+j}},\,\, x^*\in S_{X_{m+j}^*}, x^*(x)=1\}
$$
$$
\leq \sup\{ |x^* L Q_{m,j}P_{m+j} x |:\,\,x\in S_{X_{m+j+1}},\,\, x^*\in S_{X_{m+j+1}^*}, x^*(x)=1\} =w_{m+j+1}(L).
$$
Now consider $x^* \in X_{m+1}^*$ and $x \in X_{m+1}$ with $x^*(x) = \parallel x \parallel = \parallel x^* \parallel = 1$. Suppose the norm on $X$ satisfies (LCC), from the facts that $(P_mx)^* = b_m(x) x^*\mid _{X_m}$ and
$\left| \displaystyle\frac{b_m(x)}{\parallel P_mx \parallel }\right| \geq 1$ we have
$$
 | x^* L \circ P_m x | \leq \left| \displaystyle \frac{b_m(x)}{\parallel P_mx \parallel }\right|x^* L \circ P_m x |
= |(P_mx)^* L (\frac{P_m x }{\parallel P_m x \parallel}) | \leq \nu(L).
$$
Taking supremum over $x^*\in X^*_{m+1}$ and $x \in X_{m+1}$ we obtain $w_{m+1} (L)\leq \nu(L) = w_m (L)$ and thus $w_m(L)= w_{m+1}(L)$. Induction on $j$ results in $w_m(L)= w_{m+j}(L)$.
\end{proof}
\begin{pro}
Let $P_j \in \mathcal{P}(X_{j+1}, X_j) $ with $\parallel P_j \parallel = 1$. For a fixed $m \in \mathbb{N}$, define projections $Q_{m,j} \in \mathcal{P}(X_{m+j}, X_m) $ as
$ Q_{m,j} = P_m \circ P_{m+1} \circ \cdots \circ P_{m+j-1}$. Then
$$
 \displaystyle \lim_{j \rightarrow\infty} Q_{m,j} = Q_m
$$
where $Q_m \in \mathcal{P}(X, X_m) $ with $\parallel Q_m \parallel = 1$ and $X= \displaystyle\overline{ \bigcup_{m=1}^{\infty}X_m}$.
\end{pro}
\begin{proof}
Let $ x \in \displaystyle{ \bigcup_{m=1}^{\infty}X_m}$,  then there is a minimal index $k$ such that $x \in X_k$. Choose an index $j_k$ such that $m+j_k-1\geq k.$ Note that $Q_{m,j}x = Q_{m, j_k}x$ for all $j \geq j_k$. This follows from the very definition of
$$
Q_{m,j}(x) = Q_{m,j_k} \circ (P_{m+j_k} \circ \cdots \circ P_{ m+j-1})(x)
$$
and the fact that $P_{m+j_{k}}$ is a projection onto $X_{m+j_{k}-1}$ with $X_k \subset X_{m+j_{k}-1}$ implying
$$
(P_{m+j_k} \circ P_{m+j_k+1} \cdots \circ P_{ m+j-1})(x)= x
$$.
Define the limit of the almost constant sequence $\{ Q_{m,j}x\} $ as $\displaystyle \lim_{j \rightarrow\infty} Q_{m,j}(x) = Q_m(x)$ for all $ x \in \displaystyle{ \bigcup_{m=1}^{\infty}X_m}$. Since a continuous, linear map defined on a dense subspace can be uniquely extended to the whole space, we can extend $Q_m$ uniquely to $X = \dis \overline{\bigcup_{m=1}^{\infty}X_m}$. It is clear that $Q_m \in \cal{P}(X,X_m)$
and $ \| Q_m\| =1.$
\end{proof}

\begin{pro}
For a fixed $m\in \mathbb{N}$ and $L \in \mathcal{L}(X_m)$ with $\parallel L \parallel =1$, we have
$$
w_{m+j} (L) \leq w_{m,\infty} (L)
$$
for all $j$, where $w_{m, \infty} (L)= \nu(L \circ Q_m)$.
\end{pro}

\begin{proof}
Since $w_{m,\infty} (L) = \sup\{|x^*L \circ Q_m x| : \,\, x\in S_X, \,\, x^* \in S_{X^*}, \,\,x^*(x)=1\}$, it is clear that
$$
w_{m, \infty} (L) \geq \sup\{|x^*L \circ Q_m x| : \,\, x\in S_{X_{m+j}}, \,\, x^* \in S_{X^*_{m+j}}, \,\,x^*(x)=1\}
$$
and that $Q_m (x)= Q_{m,j+1}(x)$ for any $ x \in X_{m+j},$ implies $w_{m, \infty} (L)\geq w_{m+j} (L)$.
\end{proof}

Note that from
$$
\nu(L) \leq w_m(L) \leq w_{m+j}(L) \leq w_{m, \infty} (L) \leq \parallel L \parallel =1
$$
we know that the sequence $\{w_{m+j}(L)\} $ converges to some number say $z_m(L)$, consequently we have $z_m(L) \leq w_{m,\infty} (L)$.
Now we show that for any $ m \in \mathbb{N},$ and any $ L \in \mathcal{L}(X), $
$$
w_{m,\infty}(L) = z_m(L).
$$

\begin{pro}
\label{crucial}
Let $ X$ be an infinite-dimensional Banach space and let $ Y \subseteq X$ be its linear subspace whose norm-closure is equal to $X.$
Let for $x \in S_X$
$$
N(x) = \{ x^* \in B_{X^*}: x^*(x) = \|x\|=1\}.
$$
Define for $ L \in \mathcal{L}(X),$
$$
n_2(L) = \sup \{ |x^*Lx| : x^* \in S_{X^*}, x \in S_Y, x^*(y)=1\}.
$$
Then $ \nu(L) = n_2(L).$
\end{pro}
\begin{proof}
Fix $ L \in \mathcal{L}(X).$ Notice that by definitions of $ \nu(\cdot)$ and $ n_2(\cdot), $ $ \nu(L)\geq n_2(L). $
Now, assume the contrary that there exists $ L \in \mathcal{L}(X), $ $ \| L \| =1,$ such that $ \nu(L) > n_2(L).$
Fix $ \epsilon > 0,$ $ x \in S_X \setminus Y$ and $ x^* \in S_{X^*}, $ with $ x^*(x)=1,$ such that
\begin{equation}
\label{inequality}
|x^*(Lx)| > n_2(L) + \epsilon.
\end{equation}
Define
$$
W(x) = \{ z^* \in N(x): \hbox{ there exist } \{ y_{\beta}\} \subset S_Y \hbox{ and } \{y_{\beta}^*\} \subset S_{X^*}, y_{\beta}^* \in N(y_{\beta}),
$$
$$
\| x - y_{\beta}\| \rightarrow 0,
y^*_{\beta} \rightarrow z^* \hbox{ weakly}^* \hbox{ in } X^* \}.
$$
First we show that $W(x)$ is a weak-$*$ closed subset of $ B_{X^*}.$ To do this, assume that  there exists a net $ \{z^*_\beta\} \subset W(x)$
converging to $ z^* \in B_X^*.$  In particular, this shows that $ z^*(x)=1, $ since $ z^*_{\beta}(x)=1$ for any $ \beta.$ Let $ V$ be any neighbourhood of $0$ in weak-$*$ topology of $ X^*.$
Then we can find a a weak-$*$ open neighbourhood of $0$ $ W \subset X^*$ such that $ W + W \subset V.$ Note that  for $ \beta \geq \beta_o,$ $ z^*-z_{\beta}^* \in W.$  Also for any $ \beta $ we can choose $ y_{\beta} \in S_Y$ and $ y_{\beta}^* \in N(y_{\beta}),$
such that $ \|y_{\beta} - x\| < 1/n$ and $ y_{\beta}^* - z_{\beta}^* \in W $ for $ \beta \geq \beta_1.$ Since $\{ \beta\}$ is a directed set, we can choose $ \beta_2 $ greater than $\beta_o$ and $\beta_1.$ Consequently, by the choice of $W,$ $z^* - y^*_{\beta} \in V$ for any
$\beta \geq \beta_2,$ which shows that $W(x)$ is a closed set.
\newline
By  above reasoning, (\ref{inequality}) and definition of $n_2(L)$ $ x^* \notin W(x) = cl(W(x)),$ where the closure is taken with respect to the weak-$*$ topology in $ X^*.$ Define for $n \in \mathbb{N}$
$$
z_n = \frac{x + (Lx)/n}{\|x + (Lx)/n\|}.
$$
Let $ y_n \in S_Y$ be so chosen such that $ \| y_n - z_n\| < \epsilon/(4n).$ Select for any $n \in \mathbb{N},$ $ y_n^* \in  N(y_n).$
Let $y^*$ be a cluster point of $ \{ y_n^* \}$ with respect to the weak-$*$ topology.
By definition of $ W(x)$ and (\ref{inequality})
\begin{equation}
\label{important}
|y^*(Lx)| + \epsilon \leq \limsup_n |y_n^*(Ly_n)| + \epsilon \leq  n_2(L) + \epsilon < |x^*(Lx)|.
\end{equation}
Without loss of generality, replacing $L$ by $-L,$ if necessary, we can assume that  $x^*(Lx) > 0.$ By (\ref{important}) above for $n \geq n_o,$
$$
y_n^*(Lx) + (2/3)\epsilon < x^*(Lx)
$$
and consequently, since $ \|y_n^*\| \leq 1,$
$$
y_n^*(x+(Lx)/n) + (2/3)\epsilon  < x^*(x +(Lx)/n).
$$
Hence
$$
y_n^*(z_n) + \frac{2\epsilon}{3\|x+(Lx)/n\|}  < x^*(z_n).
$$
Since $ \| x + (Lx)/n\| \rightarrow \|x\|=1,$
$$
y_n^*(z_n) + \epsilon/2  < x^*(z_n),
$$
for $n \geq n_1.$
Since $ \| y_n - z_n \| < \epsilon/4,$
$$
y_n^*(y_n) = y_n^*(y_n-z_n) + y_n^*(z_n) \leq \|y_n-z_n\| + y_n^*(z_n) < \epsilon/4 + y_n^*(z_n)
$$
$$
< x^*(z_n) - \epsilon/4 \leq x^*(y_n) + \|z_n - y_n\| - \epsilon/4 < x^*(y_n).
$$
Hence $ y_n^*(y_n) < x^*(y_n) $ for $ n\geq n_o, $ which leads to a contradicion, since $ y_n^*(y_n) = \|y_n\|=1$
and $ \|x^* \|=1.$
\end{proof}
\begin{cor}
\label{interesting}
For any $m \in \mathbb{N},$ $ z_m(L) = w_{m, \infty}(L).$
\end{cor}
\begin{proof}
 Let $ Y = \bigcup_{n=1}^{\infty} X_n.$ First we show that for any $ m \in \mathbb{N}$ and any $ L \in \mathcal{L}(X_m),$
$$
z_m(L) = lim_j w_{m+j}(L) = n_2(L\circ Q_m).
$$
Fix $ \epsilon > 0.$ Then there exists $ j \in \mathbb{N}$ such that
$$
z_m(L) < w_{m+j}(L) + \epsilon = \sup \{ |x^*LQ_{m,j+1}x|, x \in S_{X_{m+j}}, x^*\in S_{X_{m+j}^*}, x^*(x)=1 \} + \epsilon
$$
$$
\leq \sup \{ |x^*LQ_{m}x|, x \in S_{Y}, x^*\in S_{X^*}, x^*(x)=1 \} + \epsilon \leq n_2(L\circ Q_m) + \epsilon,
$$
which shows that $ z_m(L) \leq n_2(L\circ Q_m).$ On the other hand, by definition of $n_2(\cdot),$
for fixed $ m \in \mathbb{N}$ $ \epsilon >0,$ there exists $ j \in \mathbb{N},$ $x \in S_{X_{m+j}},$ $ x^* \in \S_{X_{m+j}^*}$ with
$x^*(x)=1$ such that
$$
n_2(L \circ Q_m)) \leq |x^*(LQ_m)x)| + \epsilon \leq |x^*(LQ_{m,j+1})x)| + \epsilon \leq w_{m+j}(L) \leq z_m(L),
$$
which shows that $n_2(L\circ Q_m) \leq z_m(L).$
Since by Proposition(\ref{crucial}), for any $ m \in \mathbb{N},$ $ n_2(L \circ Q_m) = \nu(L\circ Q_m) = w_{m, \infty}(L), $
$ z_m(L) = w_{\infty}(L),$ which completes the proof.
\end{proof}
\begin{pro}
\label{LCC}
Assume that $ \| \cdot \|_X$ satisfies (LCC). Then for any $ m \in \mathbb{N}, $
$$
n(X_m) \geq n(X).
$$
\end{pro}
\begin{proof}
Fix $ \epsilon > 0,$  $m \in \mathbb{N}$ and choose $ L \in \mathcal{L}(X_m), $ $ \| L \| =1 $ such that $ n(X_m) + \epsilon > \nu(L).$
By (LCC), $ \nu(L) = z_m(L).$ By Corollary (\ref{interesting})
$$
z_m(L) = w_{m, \infty}(L) = \nu(L\circ Q_m) \geq n(X),
$$
since $ \| Q_m\| =1.$
We showed that $ n(X_m)+ \epsilon \geq n(X) $ for any $ \epsilon > 0.$ Hence $ n(X_m) \geq n(X),$ as required.
\end{proof}
\begin{thm}
\label{O.K}
Let $X$ and $X_m $ and $ P_m$ be as in Definition(\ref{defLCC}). Then
$$
n(X) = \displaystyle \lim_m n(X_m).
$$
\end{thm}
\begin{proof}
By Proposition (\ref{LCC}), $ n(X_m) \geq n(X)$ for any $ m \in \mathbb{N}.$ Hence,
$$
\liminf_m n(X_m) \geq n(X).
$$
By Theorem 5.1 of \cite{mm-jm-mp-br}, we already know that
$$
n(X) \leq \limsup_m n(X_m),
$$
which proves the equality.
\end{proof}
Now we present an example of a Banach space $X$ satisfying condition (LCC) from Definition(\ref{defLCC}).
\begin{exm}
\label{main}
Let for $ n \in \mathbb{N}$ $(Y_n, \| \cdot \|_n)$ be a Banach space. Set $X_1 = Y_1$ and $X_n = X_{n-1} \oplus Y_n.$
Let for $ n \in \mathbb{N},$ let $ p_n \in [1, \infty).$
Define a norm $| \cdot |_1$ on $X_1$ by $ |x|_1 = \|x\|_1$ and a norm $ |\cdot|_2$ on $X_2$ by
$$
|(x_1,x_2)|_2 = (\|x_1\|_1^{p_1}+ \|x_2\|_2^{p_1})^{1/p_1},
$$
where $x_i \in Y_i$ for $i=1,2.$ Then having defined $ |\cdot |_n$ for $x=(x_1,...,x_n) \in X_n$ we can define $ | \cdot |_{n+1}$ on
$X_{n+1}$ by
$$
| (x,x_{n+1})|_{n+1} = (|x|_n^{p_n}+\|x_{n+1}\|_{n+1}^{p_n})^{1/p_n}.
$$
Note that if $x \in X_n,$ and $ m\geq n,$ then $ |x|_m = |x|_n.$
Let
$$
F= \{ \{y_n\}: y_n \in Y_n \hbox{ and } y_n = 0 \hbox{ whenever } n \geq m \mbox{ depending on } \{y_n\} \}.
$$
One can identify $F$ with $ \bigcup_{n=1}^{\infty} X_n,$ thus enabling us to define for $x \in F,$ its norm as:
$$
\|x\|_F = \lim_n |x|_n,
$$
because for fixed $x \in F$ the sequence $ |x|_n$ is constant from some point on by the above mentionned property.
Notice that completion of $ F$ (we will denote it by $X$) is equal to the space of all sequences $\{x_n\}$ such that  $x_n \in X_n$ and
$$
 \lim_n \|Q_nx\|_F = \sup_n \|Q_nx\|_F < +\infty ,
$$
where for $n \in \mathbb{N}$ and $ x=(x_1,x_2,...)$
$$
Q_n(x) = (x_1,...,x_n,0,...).
$$
Indeed, let $ \{x^s\}$ be a Cauchy sequence in $X.$ Notice that by definition of $ \| \cdot \|_F,$ $\|Q_n|_X\| =1.$ Hence for any
$ \epsilon >0,$ there exists $ N \in \mathbb{N}$ such that for any $s,k \geq N$ and $ n \in \mathbb{N},$
$$
|Q_n(x^s-x^k)\|_n \leq \epsilon.
$$
Consequently, for any $ n \in \mathbb{N},$ $Q_n(x^s)$ converges to some point in $ X_n.$ Hence for any $ i \in \mathbb{N}$
$ (x^s)_i \rightarrow x_i \in Y_i.$  Set $x=(x_1,x_2,...).$ Then, it is easy to see that $x \in X, $ since any Cauchy sequence is bounded and
$$
\|Q_n(x)\|_F = \lim_s\|Q_n(x^s)\|_F \leq \sup_s \|x^s\|_F < +\infty .
$$
Moreover, for fixed $\epsilon >0,$ for $s,k \geq N$ and any $ n \in \mathbb{N}, $
$$
\| Q_n(x^k - x^s)\|_F  \leq  \| x^s-x^k\|_F \leq  \epsilon.
$$
Hence fixing $ k \geq N$ and taking limit over s we get  for any $ n \in \mathbb{N},$
$$
\| Q_n(x^k -x)\|_F \leq \epsilon,
$$
and consequently $ \|x-x^k\|_X \leq \epsilon $ for $ k \geq N, $ which shows that $ \{ x^k\}$ converges to $x \in X.$
Hence $ X$ is a Banach space. Since for any $ x \in X,$ $ \lim_n\|Q_n(x)-x\|=0,$  $F $ is a dense subset of $X.$
Note that for any $ n \in \mathbb{N}$ a map $P_{n}:X_{n+1} \rightarrow  X_n$ given by
$$
P_{n}(x_1,...,x_n,x_{n+1}) = (x_1,...,x_n,0),
$$
is a linear projection of norm one. By Definition(\ref{defLCC}) and the proof from Example(\ref{exCC}), the condition (LCC) is satisfied for the norm on $X.$
\end{exm}
\begin{rem}
If  for any $ n \in \mathbb{N} $ $Y_n = \mathbb{R}$ and $ p_n = p \in [1,\infty)$ then the space $X$ from Example (\ref{main}) is equal to
$l^p.$ If $p_n=p\in [1,\infty)$ for any $ n \in \mathbb{N}$ and the Banach spaces $Y_i$ are arbitrary then
$$
X = Y_1 \otimes_p Y_2 \otimes_p Y_3 \otimes_p ...
$$
If $Y_n = Y$ for any $ n \in \mathbb{N}, $ then $X = l^p(Y).$
\end{rem}

\noindent
\mbox{~~~~~~~}Asuman G\"{u}ven AKSOY\\
\mbox{~~~~~~~}Claremont McKenna College\\
\mbox{~~~~~~~}Department of Mathematics\\
\mbox{~~~~~~~}Claremont, CA  91711, USA \\
\mbox{~~~~~~~}E-mail: aaksoy@cmc.edu \\ \\
\noindent
\mbox{~~~~~~~}Grzegorz LEWICKI\\
\mbox{~~~~~~~}Jagiellonian University\\
\mbox{~~~~~~~}Department of Mathematics\\
\mbox{~~~~~~~}\L ojasiewicza 6, 30-348, Poland\\
\mbox{~~~~~~~}E-mail: Grzegorz.Lewicki@im.uj.edu.pl\\\\

\end{document}